\documentclass[12pt]{amsproc}

\usepackage{amsmath,amssymb,amsthm,mathrsfs}

\title[Simultaneous Niven Numbers for Power-Related Bases]{Simultaneous Niven Numbers in Arithmetic Progressions for Power-Related Bases}
\author[Scott Duke Kominers]{Scott Duke Kominers$^*$}
\thanks{$^*$Harvard Business School; Department of Economics and Center of Mathematical Sciences and Applications, Harvard University; and a16z crypto. \newline\indent 
I gratefully acknowledge helpful comments from Ben~Golub, Steven~J.~Miller, Ken~Ono, Jesse~Shapiro, and the editor (John Loxton), and particularly thank the anonymous referee for an especially close reading. Part of this work was conducted while I was visiting the Technological Innovation, Entrepreneurship, and Strategic Management (TIES) group at the MIT Sloan School of Management; I greatly appreciate their hospitality.\newline\indent
I used LLMs to assist with some of the computations in the preparation of this article, particularly GPT-5.2-Pro and Claude-Sonnet-4.5 (both accessed via Poe with the support of Quora, where I am an advisor). The problem, methods, and eventual written form are my own; and of course any errors remain my responsibility.}

\usepackage[usenames,dvipsnames]{xcolor}
\usepackage{geometry}
\usepackage{xspace}

\usepackage{aliascnt}
\usepackage{hyperref}
	\hypersetup{colorlinks=true, citecolor=blue, linkcolor=BrickRed, urlcolor=OliveGreen, pdfstartview={FitH}}
	
	\newtheorem{theorem}{Theorem}[section]

	\newaliascnt{lemma}{theorem}
	\newtheorem{lemma}[lemma]{Lemma}
	\aliascntresetthe{lemma}

	\newaliascnt{corollary}{theorem}
	\newtheorem{corollary}[corollary]{Corollary}
	\aliascntresetthe{corollary}

	\newaliascnt{proposition}{theorem}
	\newtheorem{proposition}[proposition]{Proposition}
	\aliascntresetthe{proposition}

	\theoremstyle{definition}
	\newaliascnt{definition}{theorem}
	
	\aliascntresetthe{definition}

	\newaliascnt{remark}{theorem}
	\newtheorem{remark}[remark]{Remark}
	\aliascntresetthe{remark}

	\newaliascnt{example}{theorem}
	
	\aliascntresetthe{example}

	\usepackage[capitalize,noabbrev]{cleveref}
	\crefname{theorem}{Theorem}{Theorems}
	\Crefname{theorem}{Theorem}{Theorems}
	\crefname{lemma}{Lemma}{Lemmas}
	\Crefname{lemma}{Lemma}{Lemmas}
	\crefname{corollary}{Corollary}{Corollaries}
	\Crefname{corollary}{Corollary}{Corollaries}
	\crefname{proposition}{Proposition}{Propositions}
	\Crefname{proposition}{Proposition}{Propositions}
	\crefname{definition}{Definition}{Definitions}
	\Crefname{definition}{Definition}{Definitions}
	\crefname{remark}{Remark}{Remarks}
	\Crefname{remark}{Remark}{Remarks}
	\crefname{example}{Example}{Examples}
	\Crefname{example}{Example}{Examples}

\newcommand{\N}{\mathbb{Z}_{\ge 1}}
\newcommand{\Z}{\mathbb{Z}}
\newcommand{\digs}{\mathsf{s}}
\DeclareMathOperator{\ord}{ord}
\DeclareMathOperator{\rad}{rad}
\DeclareMathOperator{\lcm}{lcm}
\newcommand{\Bb}{\overline{B}}
\newcommand{\zmod}[1]{\pmod{#1}}
\newcommand{\zzmod}[1]{\!\!\!\!\zmod{#1}}

\begin{document}

\begin{abstract}
Recently, Harrington, Litman, and Wong [\textit{Bulletin of the Australian Mathematical Society}, 2024] proved that every arithmetic progression contains infinitely many base-$b$ Niven numbers, for any fixed $b\ge 2$. We use a sparse repunit construction to treat a structured two-base version of the same problem, showing that every arithmetic progression with common difference relatively prime to $b$ contains infinitely many integers that are simultaneously $b$-Niven and $b^k$-Niven (indeed, we can obtain simultaneous $b^\ell$-Niven-ness for $\ell=1,\ldots, k$).
\end{abstract}

\subjclass[2020]{Primary: 11A63 / Secondary: 11B25}
\keywords{Digit sums, Niven (Harshad) numbers, arithmetic progressions, multiplicative order, repunits}

\maketitle

\section{Introduction}

A positive integer is called a \emph{Niven number} (or \emph{Harshad number}) in a given base if it is divisible by the sum of its digits in that base. These numbers were introduced in the decimal setting by Kaprekar \cite{Kaprekar1955} and have since been studied from a variety of analytic and combinatorial perspectives (see, e.g., \cite{KennedyGoodmanBest1980,kennedy1984natural,kennedy1989niven,grundman1994sequences,cai19962,dekoninck2003number,DeKoninckDoyonKatai2003,fredricksen2007remarks,de2008counting,Witno2016,Sanna2021}).

Recently, Harrington, Litman, and Wong \cite{HLW2024} proved that every arithmetic progression contains infinitely many base-$b$ Niven numbers, for any fixed $b\ge 2$. Inspired by their analysis, here we investigate a variation of their question for certain \emph{simultaneous} Niven conditions in multiple bases. 

We treat a particularly structured two-base version of the problem: considering the pair of bases $\{b,B\}$ where $B=b^k$ is a power of $b$. For a single base $b$, a multiplicative-order repunit argument already gives an elementary proof of infinitude in any arithmetic progression whose common difference is relatively prime to the base. We show how to combine that idea with a digit-sum compatibility construction---a version of Witno's ``sparse repunits''~\cite{Witno2016}---to obtain explicit simultaneous $b$- and $b^k$-Niven numbers in every such progression. As a corollary, we show that it is possible to force Niven-ness in our arithmetic progressions simultaneously for any (finite) number of powers of $b$.

The restriction to power-related bases is not merely a convenience but reflects a genuine structural simplification. When $B=b^k$, each base-$B$ digit corresponds to a block of $k$ base-$b$ digits, so by arranging for the base-$B$ digits to be small (indeed, in $\{0,1,\dots,b-1\}$) we can ensure that passing from base $B$ to base $b$ creates no carries across blocks, and hence preserves digit sums. This digit-block compatibility provides a concrete mechanism for linking Niven conditions across the two bases. Our approach exploits this linkage together with a multiplicative-order ``repunits with gaps'' construction that simultaneously forces the desired residue class modulo $m$ and divisibility by a prescribed digit sum.

\subsection{Main result}
For positive integers $m\ge 1$ and $r$ with $0\le r<m$, we write
\[
S_{m,r}:=\{n\in\N: n\equiv r \zzmod{m}\},
\]
for the arithmetic progression $r+t\cdot m$.

\begin{theorem}\label{thm:main}
Let $b \ge 2$ and $k \ge 1$ be integers, and set $B = b^k$. If $m\ge 1$ is such that $\gcd(m,b)=1$, then for any $r$ with $0\le r<m$, there exist infinitely many integers $n \in S_{m,r}$ such that
\[
\digs_b(n) \mid n
\quad\text{and}\quad
\digs_B(n) \mid n;
\]
i.e., there are infinitely many $n$ in the arithmetic progression $S_{m,r}$ that are simultaneously $b$- and $b^k$-Niven. Moreover, for any $s_0\geq1$, there exists such an $n$ with $\digs_b(n) =\digs_B(n) \geq s_0$.
\end{theorem}

\section{Preliminaries}\label{sec:prelim}

\subsection{Digit sums and base expansions}

For an integer base $g\ge 2$, every positive integer $c$ has a unique base-$g$ expansion
\[
c=\sum_{i=0}^{L} d_i g^i,\qquad d_i\in\{0,1,\dots,g-1\},\ d_L\ne 0;
\]
in this case, the sum of $c$'s base-$g$ digits is
\[
\digs_g(c)=\sum_{i=0}^{L} d_i.
\]
For convenience, we set $\digs_g(0)=0$. We say that $c$ is \emph{$g$-Niven} if $\digs_g(c)\mid c$.

We also write
\[
\rad(c):=\prod_{\substack{p\mid c\\ p\text{ prime}}} p
\]
for the \emph{radical} of an integer $c\ge 1$, i.e., the product of the distinct prime divisors of $c$, with the convention that $\rad(1)=1$.

\subsection{Compatibility between bases $b$ and $B=b^k$}

For power-related bases $b$ and $B=b^k$, a base-$B$ digit occupies a block of $k$ base-$b$ digits. The following elementary observation linking those digit blocks is central to our construction.

\begin{proposition}\label{prop:digit-compat}
Let $b\ge 2$ and $k\ge 1$ be integers, and set $B=b^k$. Suppose that $n$ admits a base-$B$
expansion of the form
\[
n=\sum_{i=0}^{L} d_i B^i
\quad\text{with}\quad
d_i\in\{0,1,\dots,b-1\}\ \text{for all $i$ and } d_L\ne 0.
\]
Then, we have
\[
\digs_b(n)=\digs_B(n)=\sum_{i=0}^{L} d_i.
\]
\end{proposition}

\begin{proof}
We write $B^i=b^{ki}$, so
\[
n=\sum_{i=0}^{L} d_i b^{ki}.
\]
Since $0\le d_i\le b-1$, the base-$b$ expansion of $d_i b^{ki}$ consists of a single digit $d_i$ in position $ki$ and $0$s elsewhere. The positions $ki$ are distinct, so when summing over $i$, each base-$b$ position receives at most one nonzero digit and therefore no carries occur. Hence the base-$b$ digit sum equals $\sum_i d_i$. On the other hand, the base-$B$ digit sum is also $\sum_i d_i$ by definition.
\end{proof}

\begin{remark}\label{rem:block-digits}
More generally, if $n=\sum_{i=0}^{L} d_i B^i$ is the base-$B$ expansion with $0\le d_i<B=b^k$, then expanding each coefficient $d_i$ in base $b$ uses at most $k$ digits, occupying the block of positions $ki,ki+1,\dots,ki+k-1$; these blocks are disjoint for different $i$, so no carries occur between them. Consequently
\[
\digs_b(n)=\sum_{i=0}^{L} \digs_b(d_i).
\]
\Cref{prop:digit-compat} is the special case in which we have $d_i<b$ for all $i$, whence we obtain $\digs_b(d_i)=d_i$.
\end{remark}

\subsection{Multiplicative order}

For integers $a\in\Z$ and $N\ge 1$ with $\gcd(a,N)=1$, the \emph{multiplicative order} of $a$ modulo $N$ is
\[
\ord_N(a):=\min\{\omega\in\N:a^\omega\equiv 1 \zzmod{N}\};
\]
$\ord_N(a)$ exists by finiteness of $(\Z/N\Z)^\times$ whenever $N\ge 2$. (When $N=1$, the congruence $a^\omega\equiv 1\zmod{1}$ holds for all $\omega$, and the above definition gives $\ord_1(a)=1$.)

We also use the Chinese Remainder Theorem (CRT) in the standard form: if $\gcd(M_1,M_2)=1$, then for any residues $x_1\zmod {M_1}$ and $x_2\zmod{M_2}$ there exists a unique residue class $x\zmod(M_1M_2)$ satisfying both congruences (see, e.g., Section 2.3 of \cite{niven1991introduction}).

\section{The construction}\label{sec:construction}

We now introduce our simultaneous-Niven number construction. Throughout, fix integers $b\ge 2$, $k\ge 1$, set $B=b^k$, and fix an arithmetic progression $S_{m,r}$ with $m\ge 1$, $0\le r<m$, and $\gcd(m,b)=1$.

\subsection*{Step 1: admissible digit sums.}
We choose integers $s\ge 1$ simultaneously satisfying
\begin{equation}\label{eq:s-admissible}
s\equiv r\zzmod{m}
\qquad\text{and}\qquad
\gcd(s,b)=1;
\end{equation}
we call such $s$ \textit{admissible}.

\subsection*{Step 2: the spacing parameter.}
Since $B=b^k$, we have $\gcd(B,ms)=1$ if and only if $\gcd(b,ms)=1$.
For each admissible $s$, the hypothesis that $\gcd(s,b)=1$ thus implies $\gcd(ms,b)=1$, given that we have chosen $m$ and $b$ such that $\gcd(m,b)=1$.
Hence, we can define\footnote{Note that any positive integer $\omega$ with $B^\omega\equiv 1\zmod{m}$ and $B^\omega\equiv 1\zmod{s}$ suffices for the congruence computations below; taking $\omega_s=\ord_{ms}(B)$ is a convenient canonical choice guaranteeing these congruences (indeed, the stronger congruence $B^{\omega_s}\equiv 1\zmod{ms}$).}
\begin{equation}
\omega_s:=\ord_{ms}(B).
\label{eq:os-def}
\end{equation}

\subsection*{Step 3: a sparse repunit in base $B$.}
Define
\begin{equation}\label{eq:ns-def}
n_s:=\sum_{j=0}^{s-1} B^{j\omega_s};
\end{equation}
equivalently, since the sum is geometric,
\[
n_s=\frac{B^{s\omega_s}-1}{B^{\omega_s}-1}.
\]

We prove in \cref{prop:ns-properties} that $n_s\equiv r\zmod{m}$, that $\digs_b(n_s)=\digs_B(n_s)=s$, and that $s\mid n_s$; i.e., for any admissible $s$, the constructed integer $n_s$ is simultaneously $b$-Niven and $B$-Niven. We then complete the proof of \cref{thm:main} in \cref{sec:infinitude} by showing that there are infinitely many admissible $s$.

\begin{remark}\label{rem:repunit}
By construction, the integer $n_s$ in \eqref{eq:ns-def} is the length-$s$ \emph{repunit} in the auxiliary base $B^{\omega_s}$, i.e.,
\[
n_s = 1+(B^{\omega_s})+(B^{\omega_s})^2+\cdots+(B^{\omega_s})^{s-1}.
\]
Equivalently, in base $B$ it is a \emph{sparse repunit}, with $1$s separated by gaps of length $\omega_s-1$. This point of view connects our construction to the study of generalized repunits by Witno~\cite{Witno2016} (see also \cite{witno2013family}). In Witno's notation (which slightly overloads with ours, so we adopt it only temporarily for this remark): For integers $n\ge 1$, $k\ge 1$, and base $b\ge 2$ we set
\[
R_{n,b,k}:=\sum_{j=0}^{n-1} b^{jk}=\frac{b^{nk}-1}{b^k-1},
\]
which in base $b$ is a string of $n$ $1$s separated by blocks of $k-1$ $0$s; equivalently $R_{n,b,k}=R_{n,b^k}$ is a generalized repunit in base $b^k$. Our $n_s$ are precisely of Witno's form with base $b$ replaced by $B$ and gap parameter $k$ replaced by $\omega_s$.

Witno's Theorem~4.2 gives a multiplicative order criterion for when such sparse repunits are $b$-Niven. In contrast, we \emph{choose} the spacing $\omega_s=\ord_{ms}(B)$ so that \[B^{\omega_s}\equiv 1\zmod{ms},\] which forces the congruence $n_s\equiv s\zmod{ms}$ and hence $s\mid n_s$, while simultaneously enforcing the arithmetic progression constraint $n_s\equiv r\zmod {m}$.
\end{remark}

\section{Verification of the desired properties}\label{sec:verification}

\begin{proposition}\label{prop:ns-properties}
Assume that $s$ satisfies \eqref{eq:s-admissible}, and define $\omega_s$ and $n_s$ by \eqref{eq:os-def} and \eqref{eq:ns-def}, respectively. Then:
\begin{enumerate}
\item $n_s\equiv r\pmod{m}$, so $n_s\in S_{m,r}$;\label{p:1}
\item $\digs_B(n_s)=\digs_b(n_s)=s$;\label{p:2}
\item $s\mid n_s$.\label{p:3}
\end{enumerate}
In particular, $n_s$ is simultaneously $b$-Niven and $B$-Niven.
\end{proposition}

\begin{proof}
As $\omega_s=\ord_{ms}(B)$, by construction, we have
$B^{\omega_s}\equiv 1\zmod{ms}$. Hence for every $j\ge 0$, we have
\begin{equation}
B^{j\omega_s}=(B^{\omega_s})^j\equiv 1\pmod{ms}.
\label{eq:os-mod}
\end{equation}
Expanding \eqref{eq:ns-def} modulo $ms$ using \eqref{eq:os-mod} gives
\begin{equation}
n_s=\sum_{j=0}^{s-1} B^{j\omega_s}\equiv \sum_{j=0}^{s-1} 1= s \pmod{ms}.
\label{eq:ns-mod}
\end{equation}
Reducing \eqref{eq:ns-mod} modulo $m$ yields $n_s\equiv s\zmod{m}$, and since $s\equiv r\zmod{m}$, we obtain $n_s\equiv r\zmod{m}$, proving Part \ref{p:1}.
Reducing \eqref{eq:ns-mod} modulo $s$ instead yields
\[
n_s\equiv s\equiv 0\pmod{s},
\]
proving Part \ref{p:3}.

For Part \ref{p:2}, in the base-$B$ expansion of $n_s$, the summands are distinct powers of $B$ with coefficient $1$, so no carries occur in base $B$: there are $1$s in positions $0,\omega_s,2\omega_s,\dots,(s-1)\omega_s$ and $0$s in all other positions. It follows that $\digs_B(n_s)=s$.
Since all base-$B$ digits of $n_s$ lie in $\{0,1\}\subseteq\{0,1,\dots,b-1\}$, \cref{prop:digit-compat} then gives $\digs_b(n_s)=\digs_B(n_s)=s$.
\end{proof}

\section{Infinitude of the admissible $s$}\label{sec:infinitude}

It remains to show that admissible $s$ exist infinitely often, and that the resulting $n_s$ are distinct.

\begin{lemma}\label{lem:many-s}
For any $b\ge 2$ and $m\ge 1$ with $\gcd(m,b)=1$, there exist infinitely many
integers $s\ge 1$ satisfying \eqref{eq:s-admissible} for any residue class $r\zmod{m}$.  Moreover, for any $s_0\geq1$, there exists such an $s$ with $s\geq s_0$.
\end{lemma}

\begin{proof}
Setting $q:=\rad(b)$, we have $\gcd(m,q)=1$ because $\gcd(m,b)=1$. By the CRT, the simultaneous congruences
\[
s\equiv r \zzmod{m}
\qquad\text{and}\qquad
s\equiv 1 \zzmod{q}
\]
have a solution $s\equiv s^\ast\zmod{mq}$. Let $s^\ast$ denote the least positive integer in that residue class. Every integer
\[
s=s^\ast+t(mq)\qquad (t\ge 0)
\]
then satisfies $s\equiv r\zmod{m}$ and $s\equiv 1\zmod{q}$; hence for such $s$ we have $\gcd(s,q)=1$. Since $q$ has the same prime divisors as $b$, this implies $\gcd(s,b)=1$. Taking $t$ arbitrarily large yields infinitely many admissible $s$, and by choosing $t$ large enough we can ensure that $s\ge s_0$.
\end{proof}

\begin{lemma}\label{lem:s-distinct}
If $s$ and $s'$ are distinct admissible integers, then $n_s\ne n_{s'}$.
\end{lemma}

\begin{proof}
We prove the contrapositive: If $n_s=n_{s'}$, then \textit{a priori},
\begin{equation}
\digs_B(n_s)=\digs_B(n_{s'}).
\label{eq:digseq}
\end{equation}
But by Part~\ref{p:2} of \cref{prop:ns-properties}, we have $\digs_B(n_s)=s$ and $\digs_B(n_{s'})=s'$; hence, \eqref{eq:digseq} implies that $s=s'$.
\end{proof}

\begin{proof}[Proof of \cref{thm:main}]
By \cref{lem:many-s}, there exist infinitely many admissible integers $s$. For each such $s$,
\cref{prop:ns-properties} shows that $n_s\in S_{m,r}$, and moreover $n_s$ is simultaneously $b$-Niven and
$B$-Niven. \Cref{lem:s-distinct} shows the $n_s$ associated to admissible $s$ are all distinct, giving infinitely many 
integers with the desired properties.

For the strengthening, given $s_0\ge 1$, we choose an admissible $s\ge s_0$ (which exists by \cref{lem:many-s}).
Part~\ref{p:2} of \cref{prop:ns-properties} then gives $\digs_b(n_s)=\digs_B(n_s)=s\ge s_0$, as desired.
\end{proof}

\section{An illustrative example}\label{sec:example}

Take $b=2$, $k=3$, so $B=b^k=8$. Let $m=5$ and $r=3$.

We choose $s\equiv 3\zmod{5}$ and $\gcd(s,2)=1$; the smallest such choice is $s=3$. We have
\begin{gather*}
8\equiv 8\pmod{15},\\
8^2\equiv 64\equiv 4\pmod{15},\\
8^4\equiv 16\equiv 1\pmod{15},
\end{gather*}
and since $8^2\not\equiv 1\zmod{15}$, we have $\omega_s=4$.

The construction then gives
\[
n_s = 1 + 8^{4} + 8^{8} = 1 + 4096 + 16777216 = 16781313.
\]
Now $n_s\equiv 3\zmod{5}$, so $n_s\in S_{5,3}$. Also, since $8^4\equiv 1\zmod{3}$, we have
\[
n_s \equiv 1+1+1 \equiv 0\pmod{3},
\]
so $3\mid n_s$.

In base $8$, we have
\[
n_s = 1\cdot 8^8 + 1\cdot 8^4 + 1\cdot 8^0 = (100010001)_8;
\]
hence, $\digs_8(n_s)=3$ as expected. In base $2$, 
\[
n_s = 1\cdot 2^{24} + 1\cdot 2^{12} + 1\cdot 2^0=(1000000000001000000000001)_2,
\]
with exactly three $1$s (at positions $0$, $12$, and $24$); hence, $\digs_2(n_s)=3$.

Thus, we see that $n_s$ is simultaneously $2$-Niven and $8$-Niven and lies in the progression $n\equiv 3\zmod{5}$.

\section{Concluding remarks}\label{sec:discussion}

We have proven that for every pair of power-related bases $\{b,b^k\}$ and every arithmetic progression
$S_{m,r}$ with $\gcd(m,b)=1$, there are infinitely many integers in $S_{m,r}$ that are simultaneously
$b$-Niven and $b^k$-Niven. The argument is elementary and rests on two simple ideas:
\begin{enumerate}
	\item digit-sum compatibility for bases $b$ and $b^k$ when base-$b^k$ digits lie in $\{0,\dots,b-1\}$, and
\item congruence forcing via multiplicative order to ensure both the prescribed progression and
divisibility by the chosen digit sum.
\end{enumerate}

\Cref{thm:main} only asks for $\digs_b(n)\mid n$ and $\digs_B(n)\mid n$, with no relationship required between
$\digs_b(n)$ and $\digs_B(n)$. Yet our construction yields the stronger property
\[
\digs_b(n_s)=\digs_B(n_s)
\]
because the base-$B$ digits are chosen to lie in $\{0,1\}\subseteq\{0,1,\dots,b-1\}$, allowing the digit-sum
compatibility obtained in \cref{prop:digit-compat}. This feature appears to be an artifact of the method,
but it is convenient and may be useful in other simultaneous-base questions.

We note also that our proof of \cref{thm:main} is completely explicit: for a chosen admissible $s$, we compute $\omega_s=\ord_{ms}(B)$ and then construct $n_s$ via \eqref{eq:ns-def}. The resulting integers grow very rapidly (roughly on the order of $B^{(s-1)\omega_s}$), so this is not an efficient enumeration method, but it does yield a concrete infinite family and a constructive existence proof.

\subsection{Relation to the order--repunit method in the coprime case}
For a fixed base $b$ and modulus coprime to $b$, using multiplicative order to force congruences is standard. In particular, for a chosen integer $s$ with $\gcd(s,b)=1$, setting
\begin{equation}
\omega_s^{(b)}:=\ord_{ms}(b),
\qquad
n_s^{(b)}:=\sum_{j=0}^{s-1} b^{j\omega_s^{(b)}}
\label{eq:cong-forcing}
\end{equation}
gives $b^{\omega_s^{(b)}}\equiv 1\zmod{ms}$, and so $n_s^{(b)}\equiv s\zmod{ms}$ and $\digs_b(n_s^{(b)})=s$. Choosing $s\equiv r\zmod{m}$ then places $n_s^{(b)}$ in a prescribed arithmetic progression. (This is precisely the special case $k=1$ of \cref{thm:main}.)

The present construction is a structured refinement of this idea, using $B$ instead of $b$ in \eqref{eq:cong-forcing}. Thus, as noted in \cref{rem:repunit}, $n_s$ is a length-$s$ repunit in the auxiliary base $B^{\omega_s}$, appearing in base $B$ as a sparse repunit (cf.~\cite{Witno2016}) with digits equal to $1$. As in the single-base--coprime case argument, multiplicative order enforces the required congruences, but here the power-related structure additionally allows the same sparse repunit to have identical digit sums in bases $b$ and $B$.

\subsection{Extension to multiple power-related bases}

If $\ell$ is a positive divisor of $k$, then the integers $n_s$ constructed in our proof of \cref{thm:main} are automatically $b^\ell$-Niven. Indeed, writing $k=\ell t$ gives
\[
n_s=\sum_{j=0}^{s-1} b^{kj\omega_s}=\sum_{j=0}^{s-1}(b^\ell)^{tj\omega_s},
\]
so the base-$b^\ell$ expansion of $n_s$ has exactly $s$ digits equal to $1$, and thus $\digs_{b^\ell}(n_s)=s$.

In our numerical example in \cref{sec:example}, we also have $n_s=(1000001000001)_4$, so that $n_s$ is $4$-Niven in addition to being $2$- and $8$-Niven. Here $4=2^2$ and $\ell=2$ does not correspond to a divisor of $k=3$, so this additional $4$-Niven-ness is an arithmetical coincidence. In general, applying the construction given for $\{b,b^k\}$ does not guarantee $b^{\ell}$-Niven-ness for intermediate exponents $\ell$ with $\ell\nmid k$.\footnote{For instance, again take $b=2$ and $k=3$ (so $B=8$), and choose $m=1$ and $s=7$. Then $\omega_s=\ord_{7}(8)=1$, so the construction gives
\[
n_s = 8^6 + \cdots + 8^2 + 8 + 1 = 299593,
\]
which is simultaneously $2$- and $8$-Niven:
\[
299593=(1001001001001001001)_2=(1111111)_8,
\]
with $\digs_2(299593)=\digs_8(299593)=7$ by construction. However, $299593=(1021021021)_4$ with $\digs_4(299593)=10\nmid 299593$, so $n_s$ is not $4$-Niven.}  Nevertheless, the divisibility observation above implies that an extension of our construction makes it possible to pick up any finite collection of power-related bases $\{b,b^2,\ldots,b^k\}$.

\begin{corollary}\label{cor:tower}
Let $b \ge 2$ and $k \ge 1$ be integers, and put $K=\lcm(1,2,\dots,k)$. If $m\ge 1$ is such that $\gcd(m,b)=1$, then for any $r$ with $0\le r<m$, there exist infinitely many integers $n \in S_{m,r}$ that are
$b^{\ell}$-Niven for all $\ell$ with $1\le \ell\le k$. Moreover, for every $s_0\ge 1$ there exists
such an $n$ with
\[
\digs_b(n)=\digs_{b^2}(n)=\cdots = \digs_{b^k}(n)\ge s_0.
\]
\end{corollary}

\begin{proof}
We fix $s_0\ge 1$ and choose an admissible $s\ge s_0$ as in \cref{lem:many-s}, so that
$s\equiv r\zmod{m}$ and $\gcd(s,b)=1$.

We set $\Bb:=b^K$ and define
\[
\Omega_s=\ord_{ms}(\Bb),
\qquad
N_s:=\sum_{j=0}^{s-1} \Bb^{j\Omega_s}
\]
(note that $\ord_{ms}(\Bb)$ exists because $\gcd(\Bb,ms)=1$). By construction, $\Bb^{\Omega_s}\equiv 1\zmod{ms}$, so the same argument as in the proof of \cref{prop:ns-properties} gives $N_s\equiv s\equiv r\zmod{m}$ and $s\mid N_s$.

Now, we fix $\ell\in\{1,2,\dots,k\}$. Since $\ell\mid K$, we can write
\[
\Bb^{j\Omega_s}=b^{Kj\Omega_s}=(b^\ell)^{(K/\ell)j\Omega_s}.
\]
Thus, we see that
\[
N_s=\sum_{j=0}^{s-1}(b^\ell)^{(K/\ell)j\Omega_s}
\]
is a sum of distinct powers of $b^\ell$, each with coefficient $1$; hence, the base-$b^\ell$ expansion of
$N_s$ has exactly $s$ digits equal to $1$ and no carries, so $\digs_{b^\ell}(N_s)=s$. Since also $s\mid N_s$,
it follows that $\digs_{b^\ell}(N_s)\mid N_s$; this holds for every $\ell=1,\dots,k$, as desired.

Finally, distinct admissible $s$ chosen in this way yield distinct integers $N_s$ since $\digs_b(N_s)=s$; the infinitude claim then follows by taking larger and larger $s_0$.
\end{proof}

\subsection{Potential for further extensions}

Our construction relies on choosing $\omega_s$ so that $B^{\omega_s}\equiv 1\zmod{m}$ and hence
$B^{j\omega_s}\equiv 1\zmod{m}$ for all $j$. \Cref{cor:tower} shows that the same construction can simultaneously handle several bases that are all powers of a common base $b$, provided the exponents are taken from a finite set. But if a prime $p$ divides both $m$ and $b$, then $p\mid B$ and so $B^\ell\equiv 0\zmod{p}$ for every $\ell\ge 1$, making it impossible for us to attain $B^{\omega_s}\equiv 1\zmod{m}$. Thus, the assumption $\gcd(m,b)=1$ is not merely technical: it is required for the particular congruence-forcing mechanism used here. 

Harrington, Litman, and Wong \cite{HLW2024}, meanwhile, proved that \textit{every} arithmetic progression---not just those with $\gcd(m,b)=1$---contains infinitely many $b$-Niven numbers for any single base $b \geq 2$; this suggests it is at least plausible that the relative primality hypothesis in \cref{thm:main} could be relaxed. We also might hope to simultaneously attain Niven-ness across sets of \textit{non}--power-related bases, but this would seem to require new ideas for how to link the structure across bases. 

\providecommand{\bysame}{\leavevmode\hbox to3em{\hrulefill}\thinspace}
\providecommand{\MR}{\relax\ifhmode\unskip\space\fi MR }
\providecommand{\MRhref}[2]{%
  \href{http://www.ams.org/mathscinet-getitem?mr=#1}{#2}
}
\providecommand{\href}[2]{#2}

\end{document}